\documentclass[12pt]{amsart}
\usepackage{amssymb,amsmath,amsthm,amssymb}
\usepackage{mathtools}
\usepackage{graphicx}
\usepackage{comment}
\usepackage{hyperref}
\usepackage{enumerate}
\usepackage{xcolor}
\usepackage{epstopdf}
\usepackage{enumerate}

\usepackage[margin=1.2in]{geometry}

\newcommand{\R}{\ensuremath{\mathbb{R}}}

\renewcommand{\epsilon}{\varepsilon}
\newcommand{\e}{\mathrm{e}}

\newcommand{\pd}{\partial}
\newcommand{\pr}{\partial}

\newcommand{\inner}[2]{\left\langle #1 \, , \, #2\right\rangle} 

\def\labelitemi{--}
\def\ba #1\ea {\begin{align} #1\end{align}}
\def\bann #1\eann {\begin{align*} #1\end{align*}}
\def\ben #1\een {\begin{enumerate} #1\end{enumerate}}
\def\bi #1\ei {\begin{itemize}\renewcommand\labelitemi{--} #1\end{itemize}}

\newtheorem{theorem}{Theorem}[section]
\newtheorem{lemma}[theorem]{Lemma}
\newtheorem{properties}[theorem]{Properties}

\newtheorem{corollary}[theorem]{Corollary}

\title[Sharp distance comparison for CSF on $S^2$]{Sharp distance comparison for curve shortening flow on the round sphere}
\author{Paul Bryan}
\address{School of Mathematical and Physical Sciences, Macquarie University, Sydney, NSW, Australia}
\email{paul.bryan@mq.edu.au}
\author{Mat Langford}
\address{Mathematical Sciences Institute, Australian National University, Canberra, ACT, Australia}
\email{mathew.langford@anu.edu.au}
\author{Jonathan J. Zhu}
\address{Department of Mathematics, University of Washington, Seattle, WA, USA}
\email{jonozhu@uw.edu}

\begin{document}

\maketitle

\begin{abstract}
We prove that curve shortening flow on the round sphere displays sharp chord-arc improvement, precisely as in the planar setting (Andrews--Bryan, Comm. Anal. Geom., 2011). As in the planar case, the sharp estimate implies control on the curvature, resulting in a direct and efficient proof that simple spherical curves either contract to round points (in finite time) or converge to great circles (in infinite time).
\end{abstract}

\section{introduction}

Curve shortening flow is the formal gradient flow of the length functional for immersed curves in Riemannian manifolds. The behaviour of simple closed planar curves under curve shortening flow is described by the theorems of Gage--Hamilton \cite{GageHamilton86} and Grayson \cite{Grayson}: any such curve must remain simple and shrink to an asymptotically round point after a finite amount of time (a number of alternative proofs have since emerged, see \cite{MR2967056,MR2843240,AndrewsBryan,MR1369140,Huisken96}). This result was extended to curve shortening flow of simple closed curves on certain Riemannian surfaces by Gage \cite{MR1046497} and Grayson \cite{MR979601} (for subsequent approaches, see \cite{MR1087347,EdelenChordarc,MR2668967}). In this setting, the curve either converges to a round point in finite time, or converges (subsequentially) to a closed geodesic as $t\to \infty$. 
If the ambient surface is the round sphere, the latter may be upgraded to full convergence: 

\begin{theorem}[
Gage \cite{MR1046497} and Grayson \cite{MR979601}]\label{thm:S2Grayson}
Let $\{\Gamma_t\}_{t\in [0,T)}$ be a maximal curve shortening flow starting from a simple regular closed curve $\Gamma_0$ on $S^2$. The curves $\Gamma_t$ are simple and either:
\begin{enumerate}
\item[(a)] $T=\infty$, in which case $\Gamma_t$ converges smoothly as $t\to \infty$ to a great circle; or 
\item[(b)] $T<\infty$, in which case $\Gamma_t$ converges uniformly to some $z\in S^2$, and
\[
\tilde\Gamma_t\doteqdot \frac{\Gamma_t -z}{\sqrt{1-\mathrm{e}^{-2(T-t)}}}
\]
converges smoothly as $t\to T$ to the unit origin-centred circle in $T_zS^2\subset \R^3$.
\end{enumerate}
\end{theorem}

Our goal is to provide a direct and efficient proof of Theorem \ref{thm:S2Grayson} using the chord-arc method introduced by Huisken \cite{Huisken96} and developed by Andrews--Bryan \cite{AndrewsBryan} (see also Edelen \cite{EdelenChordarc} and Johnson--Muraleetharan \cite{MR2668967}). Andrews and Bryan observed that it is possible to establish a \textit{sharp} estimate for the chord-arc profile under planar curve shortening flow --- sharp enough to control the curvature --- leading to a quick and direct proof of Grayson's theorem. We shall see that, remarkably, a sharp chord-arc estimate also holds for spherical curve shortening flow; in fact, the same estimate as the planar case! This estimate is again strong enough to control the curvature, leading very quickly and directly to Theorem \ref{thm:S2Grayson} (in both cases).\pagebreak

In Huisken's distance comparison argument \cite{Huisken96}, as well as the work of Edelen \cite{EdelenChordarc} and Johnson--Muraleetharan \cite{MR2668967} for curve shortening flow on surfaces, one compares the extrinsic distance (chordlength) in the ambient space to the intrinsic distance (arclength) along a simple curve $\Gamma$. It is shown that, under curve shortening flow, the chord-arc profile has a positive lower bound, which is sufficient to control the behaviour of the flow at the onset of a singularity, although via a rather more involved and indirect route. This is also the approach taken in the recent distance comparison arguments for \emph{free-boundary} curve-shortening flow due to Langford--Zhu \cite{LZ} and Ko \cite{DK}. 

A key motivation for the present work is the observation that the two asymptotic profiles in Theorem \ref{thm:S2Grayson} are both circular (in the Euclidean sense), and hence have the same chord-arc profile \textit{with respect to the Euclidean chordlength}, when $S^2$ is regarded as a subset of $\mathbb{R}^3$. This motivates us to define, in Section \ref{sec:prelim}, a spherical chord-arc profile which compares arclength to the Euclidean chordlength. In Section \ref{sec:chord-arc}, we are able to show (via a multi-point maximum principle method) that this chord-arc profile satisfies precisely the same differential inequality obtained in the planar case. This means that we are able to compare it to the Euclidean comparison profile found in Andrews--Bryan \cite{AndrewsBryan}. The resulting estimate is sharp enough to control the curvature, at which point the classification of long-time behaviour follows readily (see Sections \ref{sec:curv} and \ref{sec:conv}).

\subsection*{Acknowledgements}
P.B. was supported by the Australian Research Council (grant DP220100067). M.L. was supported by the Australian Research Council (grant DE 200101834). J.Z. was supported in part by the Australian Research Council (grant FL150100126) and the National Science Foundation (grant DMS-1802984).

\section{Preliminaries}
\label{sec:prelim}

A family $\{\Gamma_t\}_{t\in I}$ of regular curves $\Gamma_t$ on the unit two-sphere $S^2$ evolves by curve shortening flow if there exists a smooth family $\gamma:M^1\times I\to S^2$ of smooth immersions $\gamma(\cdot,t):M^1\to S^2$ of $\Gamma_t=\gamma(M^1,t)$ which satisfy
\[
\pd_t\gamma=\vec \kappa\,,
\]
where $\vec\kappa$ is the curvature vector of $\gamma$. Denoting by $T=\gamma'/\vert\gamma'\vert$ the unit tangent vector field, we choose the unit normal vector field $N$ so that $\gamma=N\times T$ (the cross product on $\mathbb{R}^3$), and define the curvature according to $\vec\kappa=-\kappa N$. If we parametrise $\gamma$ by arclength $s$, this ensures that, as a map into $\R^3$,
\[
\frac{d^2\gamma}{ds^2} =\frac{dT}{ds}= -\kappa N -\gamma\,.
\] 
In particular,
\begin{equation}\label{eq:space curvature}
\overline\kappa{}^2=1+\kappa^2\,,
\end{equation}
where $\overline\kappa$ is the space curvature.

Given $x,y\in\Gamma\subset S^2$, we define the unit vector $w(x,y)= \frac{x-y}{|x-y|}$ and the Euclidean chordlength $d(x,y)=|x-y|$. The latter is related to the spherical chordlength $\rho$ by
\begin{equation}\label{eq:spherical distance}
\cos \rho = \langle x,y\rangle = 1-\frac{d^2}{2}.
\end{equation}
Observe also that
\begin{equation}\label{eq:spherical trig identity}
\langle w,x\rangle =-\langle w,y\rangle = \frac{d}{2}.
\end{equation}

We denote by $\ell(x,y)$ the arclength between $x,y\in\Gamma$ and by $L$ the total length. At any given pair of points, we may always orient our parametrisation so that $\pr_x \ell=-1$ and $\pr_y \ell=1$, where
\[
\pr_x \ell(x,y)\doteqdot\left.\frac{d}{dh}\right|_{h=0}\ell(x+h,y)\;\;\text{and}\;\; \pr_y\ell(x,y)\doteqdot\left.\frac{d}{dh}\right|_{h=0}\ell(x,y+h)\,.
\]

We consider the chord-arc profile $\psi_\Gamma$ of $\Gamma$ relative to the \emph{Euclidean} chordlength,
\begin{equation}\label{eq:chordarc definition}
\psi_\Gamma(z):=\inf\{d(x,y):\ell(x,y)=z\}\,.
\end{equation}
Observe that all parallels on $S^2$ have chord-arc profile $\psi(z)=\frac{2}{\bar\kappa}\sin(\frac{\bar\kappa z}{2})=\frac{L}{\pi}\sin(\frac{\pi z}{L})$ according to \eqref{eq:chordarc definition} (not just arbitrarily small ones); since we must take into account great circular limits, and since a sharp estimate should take equality on shrinking parallels, this motivates consideration of the Euclidean (rather than spherical) chordlength.

\begin{lemma}[Cf. {\cite[Proposition 3.12]{EGF}}]\label{lem:chordarc expansion}
Given any simple spherical curve $\Gamma$, 
\begin{equation}\label{eq:chord arc expansion}
\psi_\Gamma(z)=z-\tfrac{1}{24}(\max_{\Gamma}\kappa^2+1)z^3+O(z^5)\;\;\text{as}\;\; z\to 0\,.
\end{equation}
\end{lemma}
\begin{proof}
Given any pair of distinct points $(x,y)$ on $\Gamma$, parametrise $\Gamma$ by arclength $s$ from a midpoint $x_0$ of $(x,y)$ and consider the function
\[
f(s):= \vert\gamma(\tfrac{s}{2})-\gamma(-\tfrac{s}{2})\vert^2\,.
\]
Differentiating the Frenet--Serret formulae,
\[
\gamma_s=T,\;\; T_s=-\kappa N-\gamma,\;\; N_s=\kappa T\,,
\]
yields
\[
f(s)=s^2-\tfrac{1}{12}\bar\kappa^2(x_0)s^4+O(s^6)\,,
\]
which implies that
\[
d(x,y)=d(\gamma(-\tfrac{\ell}{2}),\gamma(\tfrac{\ell}{2}))=\ell-\tfrac{1}{24}\bar\kappa^2(x_0)\ell^3+O(\ell^5)\,.
\]
The expansion \eqref{eq:chord arc expansion} follows since this gives an upper bound for the infimum which is certainly attained.
\end{proof}

It is instructive to observe that, for any simple spherical curve $\Gamma$, \eqref{eq:spherical distance} implies that
\bann
\max_{\Gamma}\vert\bar\kappa\vert^2s^2
\ge\left(\int_{0}^{s}\vert\bar\kappa\vert\,ds\right)^2=\left(\int_{0}^s\left\vert\frac{dT}{ds}\right\vert\,ds\right)^2\ge\left(\arccos\!\inner{T(s)}{T(0)}\right)^2
\eann
since $T$ is itself a spherical curve. Defining $\bar K:=\max_{\Gamma}\vert\bar\kappa\vert$, we may thus estimate
\bann
d(x,y)\ge{}& \inner{\gamma(\tfrac{\ell}{2})-\gamma(-\tfrac{\ell}{2})}{T(0)}=\int_{-\frac{\ell}{2}}^{\frac{\ell}{2}}\inner{T(s)}{T(0)}ds\ge\int_{-\frac{\ell}{2}}^{\frac{\ell}{2}}\cos(\bar Ks)\,ds=\tfrac{2}{\bar K}\sin\left(\tfrac{\bar K\ell}{2}\right)
\eann
with equality only if $\overline\kappa$ is constant. 
Note that $\frac{2}{\bar K}\sin(\frac{\bar Kz}{2})=z-\frac{\bar K^2}{24}z^2+O(z^5)$.

\section{The chord-arc estimate}
\label{sec:chord-arc}

We shall establish a sharp estimate for the chord-arc profile. As in \cite{AndrewsBryan}, this is achieved by preserving non-negativity of the auxiliary function
\[
Z(x,y,t):= d(x,y) - L(t) \phi\left(\frac{\ell(x,y,t)}{L(t)},t\right)
\]
for a suitable barrier $\phi:[0,1]\times [0,T)\to\R$. Here $\ell(x,y,t)$ is the arclength along $\Gamma_t$ and $L(t)$ is the total length of $\Gamma_t$. The argument will hinge on the following properties which we impose on $\phi$.
\begin{properties}\label{key properties}We assume that
\begin{enumerate}[(i)]
\item\label{key property i} $\phi(1-z) = \phi(z)$ for all $z\in[0,1]$,
\item\label{key property ii} $|\phi'|\le 1$,
\item\label{key property iii} $\phi$ is strictly concave, and 
\item\label{key property iv} given any $L>0$ such that $L\phi(z)\le 2$ for all $z\in[0,1]$, the function $z\mapsto h(L\phi(z))$ is strictly concave, where $h(d):=\arccos\left(1-\frac{d^2}{2}\right)$.
\end{enumerate}
\end{properties}

\subsection{Variation of $Z$}

The temporal variation of $Z$ is given by
\begin{equation}\label{eq:DtZ}
 \pr_t Z = \langle w, - \kappa_x N_x + \kappa_y N_y\rangle - L_t \left(\phi - \frac{\ell}{L}\phi'\right) - L \phi_t - \ell_t \phi' .
\end{equation}
The first spatial variation of $Z$ is given by
\begin{align*}
\pr_x Z = {}&\langle w,T_x\rangle + \phi',\\
\pr_y Z = {}&-\langle w, T_y\rangle -\phi',
\end{align*}
while the second is given by
\begin{align*}
\pr_x^2 Z = \frac{1}{d}(1-\langle w,T_x\rangle^2) + \langle w, -x-\kappa_x N_x\rangle - \frac{1}{L} \phi'',\\
\pr_y^2 Z = \frac{1}{d}(1-\langle w,T_y\rangle^2) - \langle w, -y-\kappa_y N_y\rangle - \frac{1}{L} \phi'',\\
\pr_x \pr_y Z = -\frac{1}{d}(\langle T_x, T_y\rangle -\langle w,T_x\rangle\langle w,T_y\rangle) + \frac{1}{L} \phi''\,.
\end{align*}
In particular, at a (spatial) critical point of $Z$,
\begin{equation}\label{eq:DZ1}
\langle w,T_x\rangle = \langle w,T_y\rangle = -\phi'
\end{equation}
and hence, recalling \eqref{eq:spherical trig identity},
\begin{subequations}\label{eq:DDZ}
\begin{align}
\pr_x^2 Z = \frac{1}{d}\left(1-\frac{d^2}{2} -\phi'^2\right) - \langle w, \kappa_x N_x\rangle - \frac{1}{L} \phi'',\\
\pr_y^2 Z = \frac{1}{d}\left(1-\frac{d^2}{2} -\phi'^2\right) + \langle w, \kappa_y N_y\rangle - \frac{1}{L} \phi'',\\
\pr_x \pr_y Z = - \frac{1}{d}\langle T_x, T_y\rangle +\frac{1}{d}\phi'^2 + \frac{1}{L} \phi'' 
\end{align}
\end{subequations}
at such a point.

\subsection{Relative configuration of tangents at the critical point}

If the critical point is a minimum with value zero, then the relative configuration of the tangent lines in space may be completely characterised (compare \cite[Lemma 3.13]{EGF} in the planar setting).
\begin{lemma}\label{lem:critical configuration}
At a zero minimum $(x,y)$ of $Z(\cdot,\cdot,t)$, if $x\neq y$ then
\begin{equation}\label{eq:DZ2}
\langle T_x,T_y\rangle = 2\phi'^2-1\,.
\end{equation}
\end{lemma}
\begin{proof}
We first deal with the possibility that $x+y=0$. In that case, $x-y$ is orthogonal to both $T_xS^2$ and $T_yS^2$. 
But then, since the critical point is a minimum, the identities \eqref{eq:DZ1} and \eqref{eq:DDZ} yield
\[
0\le (\pd_x+\pd_y)^2Z=-(1+\inner{T_x}{T_y})
\]
and hence
\[
-1=-\vert T_x\vert\vert T_y\vert\le \inner{T_x}{T_y}\le -1\,,
\]
which, by \eqref{eq:DZ1}, is the claim in this case.

So we may assume that $x+y\ne 0$. 
In particular, $x$ and $y$ are not parallel, so $\{x,y,e:=x\times y\}$ is a basis for $\mathbb{R}^3$ and we may write \[T_x = \alpha^x x + \beta^x y + \delta^x e, \qquad T_y = \alpha^y x + \beta^y y + \delta^y e.\]
Using $\langle T_x,x\rangle = \langle T_y,y\rangle=0$, $|T_x|=|T_y|=1$, and the critical point conditions 
\eqref{eq:DZ1}, one may solve for the coefficients; we obtain
\[ \alpha^x = -\left(1-\frac{d^2}{2}\right)\frac{\phi'}{d(1-\frac{d^2}{4})}  , \qquad \beta^x = \frac{\phi'}{d(1-\frac{d^2}{4})},\]
\[ \alpha^y = -\frac{\phi'}{d(1-\frac{d^2}{4})}, \qquad \beta^y = \left(1-\frac{d^2}{2}\right)\frac{\phi'}{d(1-\frac{d^2}{4})},\] 
\[ (\delta^x)^2 = (\delta^y)^2= 1- \frac{\phi'^2}{1-\frac{d^2}{4}}.\] 
%
%
Thus, depending on whether $\delta^x = \pm \delta^y$, we have
\[\langle T_x,T_y\rangle = \phi'^2 \frac{1-\frac{d^2}{2}}{1-\frac{d^2}{4}} \pm \left(1- \frac{\phi'^2}{1-\frac{d^2}{4}}\right).\]

It remains to show that only the configuration with $\delta^x =- \delta^y$ is admissible at a zero minimum of $Z$. The argument is inspired by the planar case \cite{AndrewsBryan} (see \cite[Lemma 3.13]{EGF}). We first claim that $\langle N_x,w\rangle$ and $\langle N_y, w\rangle$ have opposite signs. Indeed, suppose that $\langle N_x,w\rangle$ and $\langle N_y, w\rangle$ have the same sign. In that case, since $\Gamma_t$ bounds a disk, the minimising great circular arc $\sigma:[0,1]\to S^2$ connecting $y$ to $x$ must contain another point $u$ of $\Gamma_t$ (note that $\langle N_x,w\rangle$ and $\langle N_y, w\rangle$ have the same signs as $\langle N_x,\sigma'(1)\rangle$ and $\langle N_y, \sigma'(0)\rangle$, respectively). Since $\ell(x,u)  +\ell(u,y) =\ell(x,y)$ and $\rho(x,u) +\rho(u,y) = \rho(x,y)$, the strict concavity condition Properties \ref{key properties} \eqref{key property iv} ensures that 
\bann
h(d(x,u))+h(d(u,y))
={}&h(d(x,y))\\
={}&h\left(L\phi\left(\frac{\ell(x,y)}{L}\right)\right)\\
={}&h\left(L\phi\left(\frac{\ell(x,u)+\ell(u,y)}{L}\right)\right)\\
<{}&h\left(L\phi\left(\frac{\ell(x,u)}{L}\right)\right)+h\left(L\phi\left(\frac{\ell(u,y)}{L}\right)\right).
\eann
Therefore either $Z(x,u,t)$ or $Z(u,y,t)$ is strictly less than $Z(x,y,t)$, which contradicts the assumption that the minimum of $Z$ occurs at $(x,y,t)$. So $\langle N_x,w\rangle$ and $\langle N_y, w\rangle$ must indeed have opposite signs. 

Now recall that
\[
N_x = T_x\times x= -\beta^x e + \delta^x (e\times x),
\]
which implies that $d\langle N_x,w\rangle = -\delta^x\det(e,x,y)$. By similar reasoning, we also find that $d\langle N_y,w\rangle = -\delta^y\det(e,x,y)$. This yields $\delta^x=-\delta^y$, which completes the proof. 
\end{proof}

\subsection{The differential inequality}

Recall that $\bar{\kappa}$ denotes the space curvature; by \eqref{eq:space curvature}, we may estimate
\begin{align}
\label{eq:dLdt}
-L_t ={}& \int_\gamma \kappa^2 ds=-L + \int_\gamma \bar{\kappa}^2 ds \geq -L + \frac{1}{L} \left( \int_\gamma |\bar{\kappa}| ds\right)^2
\end{align}
(with strict inequality unless $\overline\kappa$, and hence also $\kappa$, is constant) and, similarly,
\begin{align*}
-\ell_t ={}& \int_{[x:y]}\kappa^2 ds=-\ell +\int_{[x:y]}\bar\kappa^2 ds\geq -\ell +\frac{1}{\ell} \left( \int_{[x:y]} |\bar{\kappa}| ds\right)^2\,,
\end{align*}
where $[x:y]$ denotes the shorter of the two portions of $\gamma$ that join $x$ to $y$.

By Fenchel's theorem, we may estimate $\int_\gamma |\bar{\kappa}| ds \geq 2\pi$, and so
\begin{equation}\label{eq:DtL}
L_t \leq L - \frac{4\pi^2}{L}.
\end{equation} 
On the other hand, the unit tangent $T$ defines a curve in $S^2$ joining $T_x$ to $T_y$, which has speed $|\bar{\kappa}|$ (when parametrised by the arclength on $\gamma$); since the arclength $\int_{[x:y]} |\bar{\kappa}| ds$ of $T$ is bounded below by the $S^2$-distance $\rho(T_x,T_y)$, it follows from \eqref{eq:spherical distance} that
\begin{equation}\label{eq:Dtell}
\ell_t \leq \ell - \frac{1}{\ell} (\arccos\langle T_x,T_y\rangle)^2.
\end{equation}

Suppose now that $Z$ reaches zero at some pair of off-diagonal points $x\neq y$ at some first time $t>0$. Applying \eqref{eq:DtL} and \eqref{eq:Dtell} to the temporal variation \eqref{eq:DtZ}, the gradient condition \eqref{eq:DZ2} to the second spatial variation \eqref{eq:DDZ}, and combining the results, we obtain
\begin{align*}
0\geq{}& \pr_t Z - (\pr_x - \pr_y)^2Z\\
\ge{}&\left(\frac{4\pi^2}{L}-L\right) \left(\phi - \frac{\ell}{L}\phi'\right) - L\phi_t +\left(\frac{1}{\ell}(\arccos(2\phi'^2-1))^2 - \ell\right) \phi' +d + 4 \frac{\phi''}{L}
\end{align*}
at $(x,y,t)$ (note that $\phi(z)\ge z\phi'(z)$ due to concavity and non-negativity of $\phi$). Using $d-L\phi=Z=0$ and $\arccos(2\phi'^2-1)=2\arccos(\phi')$, and estimating $(\arccos(\phi'))^2$ as in \cite[Lemma 3.15]{EGF}, we arrive at the inequality
\begin{align}\label{eq:comparison principle}
L^2\phi_t
\ge{}&4(\phi''+\pi^2\phi)+\frac{8\pi\phi'}{\tan(\pi\frac{\ell}{L})}-\frac{8\pi\phi'^2}{\sin(\pi\frac{\ell}{L})}
\end{align}
at $(x,y,t)$, with strict inequality unless the curve is a Euclidean circle (due to the strict concavity of $\phi$ and the application of H\"older's inequality to obtain \eqref{eq:DtL}). This is \textit{precisely} the inequality arrived at in the planar case (see \cite[(3.32)]{EGF})! 
 
\subsection{Completing the proof of the chord-arc estimate}

It remains only to show that the Euclidean barrier,
\[\phi(z):=a^{-1}\arctan(\tfrac{a}{\pi}\sin(\pi z))\,,\]
satisfies \eqref{key property iv} of Properties \ref{key properties}.

\begin{lemma}\label{lem:phi conditions}
The function $\phi(z):=a^{-1}\arctan(\frac{a}{\pi}\sin(\pi z))$ satisfies Properties \ref{key properties}. 
\end{lemma}
\begin{proof}
The first three properties are clear. To verify the final property, we set $\phi_L := L\phi$ and will show explicitly that $(h\circ \phi_L)''<0$ for $z\in(0,1)$ (assuming $L\phi \le 2$). 

To compute $(h\circ \phi_L)''$, it is convenient to write (following \cite[\S 3.4]{EGF}) $\phi = \Psi \circ C$, where $\Psi(\zeta):= a^{-1} \arctan(a\zeta)$, $C(z):= \frac{1}{\pi}\sin(\pi z)$. Note that the function $\Psi$ satisfies $\Psi' = \frac{1}{1+a^2\zeta^2} = \frac{1}{1+\tan^2(a\Psi)}$ and $\Psi'' = -\frac{2}{\zeta} \Psi'(1-\Psi') = -\frac{2a \tan(a\Psi)}{(1+\tan^2(a\Psi))^2}$, and the function $C$ satisfies $C''=-\pi^2 C$ and $(C')^2 = 1-\pi^2 C^2$. Thus, 
\[ (\phi')^2 = \frac{1- \frac{\pi^2}{a^2}\tan^2(a\phi) }{(1+\tan^2(a\phi))^2}\,,\;\;\text{and}\;\; \phi'' = a^{-1}\tan(a\phi) \frac{\pi^2\tan(a\phi)^2 - \pi^2-2a^2}{(1+\tan^2(a\phi))^2}.\]

Recalling that $h(d) = \arccos\left(1-\frac{d^2}{2}\right)$, we thus obtain
\ba
(h\circ\phi_L)''(z)
={}&h'(\phi_L(z))\left(\frac{\phi_L(z)(\phi_L'(z))^2}{4-\phi_L^2(z)}+\phi_L''(z)\right)\nonumber\\
={}&Lh'(\phi_L(z))\left(\frac{L^2 \phi(z)(\phi'(z))^2}{4-L^2\phi^2(z)}+\phi''(z)\right)\nonumber\\
={}& \frac{1}{L}h'(\phi_L(z)) \frac{\pi^2 F(\phi(z))}{(1+\tan^2(a\phi(z)))^2(4-L^2 \phi(z)^2)}
,\label{eq:check phi epsilon}
\ea
where $F$ is defined by
\ba
F(\phi)
={}&(\tfrac{1}{\pi^2}-\tfrac{1}{a^2}\tan^2(a\phi))\!\big(L^2-a^2(4-L^2\phi^2)\tfrac{\tan(a\phi)}{a\phi}\big)\!-\!(1+\tfrac{a^2}{\pi^2})(4-L^2\phi^2)\tfrac{\tan(a\phi)}{a\phi}.\label{eq:F for concave}
\ea

Note that $\phi$ is increasing for $z\in [0,\frac{1}{2}]$, so to prove strict concavity it is enough to show that $F(\phi)<0$ for $\phi\in (0,\frac{1}{a}\arctan(\frac{a}{\pi}))$. The argument will split into two cases, depending on the total length: 

\textbf{Case 1}: $L\leq2\pi$. Note that $\phi < \frac{1}{a}\arctan(\frac{a}{\pi}) \leq \lim_{a\to 0}\frac{1}{a}\arctan(\frac{a}{\pi})=\frac{1}{\pi}$, so in particular $L^2 \phi^2 <4$. 

Now using $\tan X\leq X$ for $X\geq 0$, we have $\frac{1}{a^2}\tan^2(a\phi) \leq \phi^2 <\frac{1}{\pi^2}$, and so
\[
F(\phi)\le\left(\tfrac{1}{\pi^2}-\tfrac{1}{a^2}\tan^2(a\phi)\right)(L^2-a^2(4-L^2\phi^2))-(1+\tfrac{a^2}{\pi^2})(4-L^2\phi^2)\,.
\]
If $L^2-a^2(4-L^2\phi^2)\leq 0$, then $F(\phi)<0$ as the second term is strictly negative. Otherwise, we can estimate
\bann
F(\phi)\le{}&\left(\tfrac{1}{\pi^2}-\phi^2\right)(L^2-a^2(4-L^2\phi^2))-(1+\tfrac{a^2}{\pi^2})(4-L^2\phi^2)\\
={}&\tfrac{L^2}{\pi^2}-4-\tfrac{a^2}{\pi^2}(4-L^2\phi^2)(2-\pi^2\phi^2)\\
\le{}&\tfrac{L^2}{\pi^2}-4 <0.
\eann

\textbf{Case 2:} $L>2\pi$. Define $a_0>0$ and $\phi_0$ by
\[
\phi_0:=\tfrac{2}{L}=:\tfrac{1}{a_0}\arctan(\tfrac{a_0}{\pi})\,.
\]

The supposition that $L\phi(z) \leq 2$ for all $z\in[0,\frac{1}{2}]$ implies that $a\ge a_0$. In this case, it is convenient to consider the quotient of the two terms in (\ref{eq:F for concave}), rather than their difference $F$. To see that this quotient is less than one, it suffices, by the definitions of $a_0$ and $\phi_0$, to show that the function $(a_0,a,\phi_0,\phi)\mapsto Q_{(a_0,a)}(\phi_0,\phi)$ defined by
\[
\begin{split}
 Q_{(a,a_0)}(\phi,\phi_0):=\frac{\left(\frac{1}{a_0^2}\tan^2(a_0\phi_0)-\frac{1}{a^2}\tan^2(a\phi)\right)\left(1-a^2\left(\phi_0^2-\phi^2\right)\frac{\tan(a\phi)}{a\phi}\right)}{\left(1+\frac{a^2}{\pi^2}\right)\left(\phi_0^2-\phi^2\right)\frac{\tan(a\phi)}{a\phi}}
\end{split}
\]
is less than 1 for all $0\le \phi<\phi_0<\frac{1}{\pi}$ and $0<a_0<a$. In fact, since $Q$ is decreasing in $a$, it suffices to prove that $Q_{(a_0,a_0)}(\phi_0,\phi)<1$ for $0\le \phi<\phi_0<\frac{1}{\pi}$. 


We calculate (using $1+\tfrac{a_0^2}{\pi^2} = 1+\tan^2(a_0\phi_0)$) that
\bann
Q_{(a_0,a_0)}(\phi_0,\phi)={}&\frac{\frac{1}{a_0^2}\tan^2(a_0\phi_0)-\frac{1}{a_0^2}\tan^2(a_0\phi)}{\left(1+\tan^2(a_0\phi_0)\right)\left(\phi_0^2-\phi^2\right)\frac{\tan(a_0\phi)}{a_0\phi}}-\frac{\tan^2(a_0\phi_0)-\tan^2(a_0\phi)}{1+\tan^2(a_0\phi_0)}.
\eann
The last term is nonpositive, so substituting $X=\tan(a_0\phi)$, $Y=\tan(a_0\phi_0)$, we have 
\[
Q_{(a_0,a_0)}(\phi,\phi_0) \leq  \frac{1}{1+Y^2}\frac{Y^2-X^2}{(\arctan Y)^2-(\arctan X)^2}\frac{\arctan(X)}{X}. 
\]
It thus suffices to establish the inequality
\begin{equation}\label{eq:TBP}
q(X,Y):=\frac{1}{1+Y^2}\frac{Y^2-X^2}{(\arctan Y)^2-(\arctan X)^2}\frac{\arctan(X)}{X}<1
\end{equation}
for any $0\le X<Y<\infty$. To prove \eqref{eq:TBP}, we first observe that it holds at the extremes:
\begin{equation}\label{eq:extremes are good}
\lim_{Y\searrow X}q(X,Y)=1 \;\;\text{and}\;\; \lim_{Y\to\infty}q(X,Y)=\frac{\arctan(X)}{X(\frac{\pi^2}{4}-(\arctan X)^2)}<1\,.
\end{equation}
The limiting values are clear. To establish the inequality for the second limit, set $X=\cos\theta$ and estimate $\sin\theta<\theta$ and $\cos\theta\ge\frac{2}{\pi}(\frac{\pi}{2}-\theta)$ to obtain
\[
\frac{\arctan(X)}{X(\frac{\pi^2}{4}-(\arctan X)^2)}=\frac{(\frac{\pi}{2}-\theta)\tan\theta}{\theta(\pi-\theta)}\le\frac{\frac{\pi}{2}}{\pi-\theta}<1\,.
\]

Due to \eqref{eq:extremes are good}, for any $X\ge 0$ there is some $Y\in[X,\infty)$ at which $Y\mapsto q(X,Y)$ attains a global maximum. If $Y>X$ at the maximum, then (note that $q>0$)
\[
0=\frac{\pr}{\pr Y}\log q=-\frac{2Y}{1+Y^2}+\frac{2Y}{Y^2-X^2}-\frac{2\arctan Y}{(1+Y^2)((\arctan Y)^2-(\arctan X)^2)}\,,
\]
and hence
\[
q(X,Y)=\frac{\frac{1}{1+Y^2}\frac{Y}{\arctan Y}}{\frac{1}{1+X^2}\frac{X}{\arctan X}}\,.
\]

But the right-hand side is strictly less than $1$ when $Y>X$, since the function $h(X):=\frac{1}{1+X^2}\frac{X}{\arctan X}$ satisfies
\[
(\log h)'=\frac{1}{X}-\frac{1}{(1+X^2)\arctan X}-\frac{2X}{1+X^2}\le \frac{1}{X}-\frac{1}{(1+X^2)X}-\frac{2X}{1+X^2}=\frac{-X}{1+X^2}<0
\]
for $X>0$. This establishes \eqref{eq:TBP}, and thereby completes the proof of the lemma.
\end{proof}

Putting everything together, we can now establish the desired chord-arc estimate.

\begin{theorem}\label{thm:chordarc estimate}
If  
$\frac{d}{L} \geq \frac{1}{a}\arctan\left(\frac{a}{\pi}\sin\left(\frac{\pi\ell}{L}\right)\right)$ 
at $t=0$ for some $a>0$, then
\begin{equation}\label{eq:chord arc estimate}
\frac{d}{L} \geq \frac{1}{a \e^{-4\pi^2 \tau}} \arctan\left(\frac{a\e^{-4\pi^2 \tau}}{\pi}\sin \left(\frac{\pi\ell}{L}\right)\right)\;\;\text{at all}\;\; t\in[0,T)\,,
\end{equation}
where $\tau(t) := \int_0^t L^{-2} dt$.
\end{theorem}
\begin{proof}
We may assume that $\Gamma_0$ is not a parallel in $S^2$. Consider, for any $\varepsilon\in(0,1)$, the function $\phi_\varepsilon:=(1-\varepsilon)\phi$. Observe that $\phi_\varepsilon$ still satisfies Properties \ref{key properties}. Indeed, properties \eqref{key property i}-\eqref{key property iii} require only cursory inspection and property \eqref{key property iv} for $\phi_\epsilon$ is follows from property \eqref{key property iv} for $\phi$ (applied with $L \mapsto (1-\epsilon)L$).

Since $\phi_\varepsilon$ satisfies the strict inequality $\phi_\varepsilon'(0)<1$, the difference $Z_\varepsilon:=d-L\phi_\varepsilon(\frac{\ell}{L},\cdot)$ is strictly positive near the diagonal on any simple spherical curve (since $d(-\frac{s}{2},\frac{s}{2})=s+o(s)$ as $s\to 0$). Thus, if $Z_\varepsilon$ ever attains a strictly negative value, then it must have first reached zero at some off-diagonal point $(x,y)$ at some interior time $t$. In particular, $\Gamma_t$ cannot be a parallel, in which case we have established that the inequality \eqref{eq:comparison principle} holds strictly at $(x,y,t)$. But this is impossible: $\phi$ satisfies \eqref{eq:comparison principle} with equality (see \cite[p. 80]{EGF}), and it follows that $\phi_\varepsilon$ satisfies the reverse inequality of \eqref{eq:comparison principle}.
\end{proof}

\section{The curvature estimate}
\label{sec:curv}

As in the planar case, the sharp chord-arc estimate implies a bound on the curvature: 

\begin{corollary}
If $\frac{d}{L} \geq \frac{1}{a}\arctan\left(\frac{a}{\pi}\sin\left(\frac{\pi\ell}{L}\right)\right)$ at $t=0$ for some $a>0$, then
\begin{equation}\label{eq:curvature bound}
\kappa^2+1\le\left(\frac{2\pi}{L}\right)^2\left(1+\tfrac{2a^2}{\pi^2}\mathrm{e}^{-8\pi^2\tau}\right) \;\;\text{at all}\;\; t\in[0,T)\,.
\end{equation}
\end{corollary}
\begin{proof}
Using \eqref{eq:chord arc estimate}, compare \eqref{eq:chord arc expansion} with
\[
\tfrac{1}{a}\arctan\left(\tfrac{a}{\pi}\sin(\pi z)\right)=z-\left(\tfrac{a^2}{3}+\tfrac{\pi^2}{6}\right)z^3+O(z^5).\qedhere
\]
\end{proof}

\section{Convergence to a round point or a great circle}
\label{sec:conv}

We conclude by outlining how the curvature estimate \eqref{eq:curvature bound} leads quickly and directly to Theorem \ref{thm:S2Grayson}. The argument is very similar to that of the planar case 
(as presented, e.g., in \cite[\S 3.5]{EGF}.)

Note first that Lemma \ref{lem:chordarc expansion} ensures that some $a>0$ satisfying the hypothesis of Theorem \ref{thm:chordarc estimate} can always be found when $\Gamma_0$ is of class $C^2$. (If $\Gamma_0$ is less regular, but a suitable existence result is still available, then we may simply wait a short time $\delta$ for the flow to smooth $\Gamma_0$, while still remaining simple, and proceed with $\Gamma_\delta$ replacing $\Gamma_0$.)


\subsection{Finite maximal time} In this case, a well-known argument employing Bernstein-type estimates for the derivatives of $\kappa$ ensures that $\limsup_{t\to T}\max_{\Gamma_t}\kappa^2=\infty$ (see, e.g., \cite[Lemma  2.11]{MR1046497}). So the curvature estimate \eqref{eq:curvature bound} and the monotonicity of $L$ imply that $L(t)\to 0$ as $t\to T$. In particular, since the diameter is controlled from above by $L$, we see that each $\Gamma_t$ must lie in some ball $B_{r(t)}(z(t))$ such that $r(t) \to 0$ as $t\to T$. As the avoidance principle then also ensures that $\Gamma_t$ is contained in $B_{r(t_0)}(z(t_0))$ for every $t\geq t_0$,  we must have $\Gamma_t\to z$ as $t\to T$ for some $z\in S^2$. 

Applying the curvature estimate \eqref{eq:curvature bound} to the first variation for length \eqref{eq:dLdt}, and then using H\"older's inequality and Fenchel's theorem gives
\[
L^2-4\pi^2\left(1+\tfrac{2a^2}{\pi^2}\right) \le \frac{1}{2}\frac{d}{dt}L^2 \le L^2-4\pi^2. 
\]
As $\lim_{t\to T}L(t)=0$, integrating these inequalities yields the inequalities
\[
2\pi\sqrt{1-\mathrm{e}^{-2(T-t)}}\le L\le  2\pi\sqrt{\left(1+\tfrac{2a^2}{\pi^2}\right)\left(1-\mathrm{e}^{-2(T-t)}\right)}\,.
\]
This in turn gives an estimate for $\tau=\int_0^t L^{-2} dt$:
\[
-\frac{1}{8\pi^2+16a^2}\log\left(\frac{e^{2(T-t)}-1}{e^{2T}-1}\right)\le\tau\le-\frac{1}{8\pi^2}\log\left(\frac{e^{2(T-t)}-1}{e^{2T}-1}\right)\,.
\]
In particular, $\tau\to\infty$ as $t\to T$. Feeding this estimate for $\tau$ back into \eqref{eq:curvature bound} gives
\ba
\kappa^2+1\le
{}& \left(\tfrac{2\pi}{L}\right)^2\left(1+C\left(e^{2(T-t)}-1\right)^{\delta}\right)\,,\label{eq:kappa decay finite time}
\ea
where $\delta:=\frac{1}{1+\frac{2a^2}{\pi^2}}$ and $C:=\frac{2a^2}{4\pi^2}\left(e^{2T}-1\right)^{-\delta}$. Applying \eqref{eq:kappa decay finite time} to \eqref{eq:dLdt} to again estimate $\frac{d}{dt}L^2$ yields an improved upper bound for $L$: 
\ba\label{eq:L decay finite time}
L^2\le
{}& 4\pi^2\left(1-\mathrm{e}^{-2(T-t)}\right)\left(1+\tfrac{C}{\delta+1}(\mathrm{e}^{2(T-t)}-1)^{\delta}\right)\,;
\ea
feeding this into \eqref{eq:kappa decay finite time} gives, in particular, that $(1-\mathrm{e}^{-2(T-t)})\kappa^2\le C_0^2$. 

Note that $\sqrt{1-\e^{-2(T-t)}} \kappa$ is precisely the curvature of the rescaled curve $\frac{\Gamma_t-z}{\sqrt{1-\e^{-2(T-t)}}}$. By the (time interior) Bernstein estimates mentioned above, the bound on rescaled curvature also implies bounds on all derivatives of (rescaled) curvature. 


To see that the rescaled curvature converges to 1, consider
\[\frac{1}{L}\int\left\vert\sqrt{1-\mathrm{e}^{-2(T-t)}}\vert\bar\kappa\vert-1\right\vert^2={}\frac{1}{L}\int\left((1-\mathrm{e}^{-2(T-t)})(\kappa^2+1)-2\sqrt{1-\mathrm{e}^{-2(T-t)}}\vert\bar\kappa\vert+1\right).\]
Applying \eqref{eq:kappa decay finite time} to the first term, Fenchel's theorem to the second, and then recalling the length estimates \eqref{eq:L decay finite time} and $L^2\ge 4\pi^2(1-\mathrm{e}^{-2(T-t)})$, we find that 
\bann
\frac{1}{L}\int\left\vert\sqrt{1-\mathrm{e}^{-2(T-t)}}\vert\bar\kappa\vert-1\right\vert^2
\le{}&3C(\mathrm{e}^{2(T-t)}-1)^\delta\,.
\eann

Since $\kappa_s$ is bounded after rescaling, interpolation now yields the desired decay estimate
\[
\left\vert\sqrt{1-\mathrm{e}^{-2(T-t)}}|\bar\kappa|-1\right\vert^2\le C(\mathrm{e}^{2(T-t)}-1)^\delta,
\]
with possibly worse constants $C$ and $\delta$ (though the exponent can be improved by iterating the bootstrapping process). Decay estimates for all derivatives of (rescaled) curvature now also follow by interpolation. 

Since $\kappa$ is the (normal) speed of the curve shortening flow, the decay estimates for rescaled curvature and its derivatives may be converted into estimates for the rescaled position vector and its derivatives (cf. \cite[Lemmas 2.14 and 2.16]{EGF}; note that curvature and its derivatives control torsion and its derivatives for curves on the sphere
). These imply smooth convergence of the rescaled curves $\frac{\Gamma_t-z}{\sqrt{1-\e^{-2(T-t)}}}$ to the unit circle in $T_zS^2$.

\subsection{Infinite maximal time} In this case, the monotonicity of $L$ (recall \eqref{eq:dLdt}) ensures that $L\to L_\infty$; as the lifespan is infinite, $L_\infty$ must be at least $2\pi$ since, by \eqref{eq:DtL},
\[
\frac{d}{dt}L^2\le 2(L^2-4\pi^2)\,.
\]
In particular, $\tau(t)\doteqdot\int_0^t\frac{dt}{L^2}$ remains comparable to $t$ as $t\to\infty$, so the curvature estimate \eqref{eq:curvature bound} implies that $L_\infty= 2\pi$ and that $\kappa^2$ decays exponentially in $t$ as $t\to\infty$.

The Bernstein estimates now provide bounds on all derivatives of $\kappa$, and the exponential decay of $\kappa$ implies, by interpolation, that those derivatives also decay exponentially in time.

Since $\kappa$ is the (normal) speed of the curve shortening flow, the exponential decay estimates for $\kappa$ and its derivatives may be converted into estimates for the position vector and its derivatives, which imply smooth convergence to a great circle in $S^2$.

\bibliographystyle{acm}
\bibliography{../bibliography}

\end{document}